\documentclass[11pt]{amsart}
\usepackage{amssymb}
\usepackage{latexsym}
\usepackage{amsmath}
\usepackage{amsfonts}
\usepackage[colorlinks]{hyperref}
\usepackage{enumitem}
\usepackage{fancyhdr}
\usepackage{cleveref}
\usepackage{color}
\theoremstyle{plain}
\newtheorem{theorem}{Theorem}[section]
\newtheorem{untheorem}{Theorem}

\newcommand{\R}{\ensuremath{\mathbb{R_+}}}
\newtheorem{corollary}{Corollary}[theorem]
\newtheorem{lemma}[theorem]{Lemma}
\newtheorem{proposition}[theorem]{Proposition}

\newtheoremstyle{remark}
    {\dimexpr\topsep/2\relax} 
    {\dimexpr\topsep/2\relax} 
    {}          
    {}          
    {\bfseries} 
    {.}         
    {.5em}      
    {}          

\theoremstyle{remark}
\newtheorem{remark}{\emph{\textbf{Remark}}}[section]

\newtheoremstyle{example}
    {\dimexpr\topsep/2\relax} 
    {\dimexpr\topsep/2\relax} 
    {}          
    {}          
    {\bfseries} 
    {.}         
    {.5em}      
    {}          

\theoremstyle{example}
\newtheorem{example}{\emph{\textbf{Example}}}[section]

\newtheoremstyle{definition}
    {\dimexpr\topsep/2\relax} 
    {\dimexpr\topsep/2\relax} 
    {}          
    {}          
    {\bfseries} 
    {.}         
    {.5em}      
    {}          

\theoremstyle{definition}

\newtheorem*{similartheorem*}{Theorem \dualnumber{$'$}}

\numberwithin{equation}{section}

\begin{document}
\title{Dilation-commuting operators on power-weighted Orlicz classes}
{\let\thefootnote\relax}

%
\author{Ron Kerman}
\author{Rama Rawat}
\author{Rajesh K. Singh}
%

\address{Ron Kerman: Department of Mathematics, Brock University, St. Catharines, Ontario, L2S 3A1, Canada}
\email{rkerman@brocku.ca}
\address{Rama Rawat: Department of Mathematics and Statistics, Indian Institute of Technology, Kanpur-208016}
\email{rrawat@iitk.ac.in}
\address{Rajesh K. Singh: Department of Mathematics and Statistics, Indian Institute of Technology, Kanpur-208016}
\email{rajeshks@iitk.ac.in}

\pagestyle{headings}

\begin{abstract}Let $\Phi_1$ and $\Phi_2$ be nondecreasing functions from $\mathbb{R_+}=(0,\infty)$ onto itself. For $i=1,2$ and $\gamma \in \mathbb{R}$, define the Orlicz class $L_{\Phi_{i}}(\mathbb{R_+})$ to be the set of Lebesgue-measurable functions $f$ on $\mathbb{R_+}$ such that
\begin{equation*}
\int_{\mathbb{R_+}} \Phi_{i} \left( k|(Tf)(t)| \right) t^{\gamma}dt < \infty
\end{equation*} 
for some $k>0$.

Our goal in this paper is to find conditions on $\Phi_1$, $\Phi_2$, $\gamma$ and an operator $T$ so that the assertions
\begin{equation}
T : L_{\Phi_2,t^{\gamma}}(\mathbb{R_+}) \rightarrow L_{\Phi_1,t^{\gamma}}(\mathbb{R_+}), \tag{I}
\end{equation}
and
\begin{equation}\label{modularA}
\int_{\mathbb{R_+}} \Phi_1 \left( |(Tf)(t)| \right)t^{\gamma}dt \leq K \int_{\mathbb{R_+}} \Phi_2 \left( K|f(s)| \right)s^{\gamma}ds, \tag{M}
\end{equation}
in which $K>0$ is independent of $f$, say, simple on $\mathbb{R_+}$, are equivalent and to then find necessary and sufficient conditions in order that (\ref{modularA}) holds. 
\end{abstract}

\maketitle

\section{Introduction}
Let the operator $T$ map  the set, $S(\mathbb{R_+})$, of simple, Lebesgue-measurable functions on $\mathbb{R_+}=(0,\infty)$ into $M(\mathbb{R_+})$, the class of Lebesgue-measurable functions on $\mathbb{R_+}$. Suppose that $T$ is positively homogeneous in the sense that
\begin{equation*}
| T(cf) | = |c| |Tf|, \ f \in S(\mathbb{R_+}), \ c \in \mathbb{R},
\end{equation*}
with, moreover, 
\begin{equation*}
(Tf)(\lambda t ) = T \left( f(\lambda \ \cdot) \right)(t), \  \lambda,t \in \mathbb{R_+}.
\end{equation*}
We call such a $T$ a \emph{dilation-commuting operator}.

Our aim in this paper is to determine when certain dilation-commuting operators map functions in a so-called Orlicz class, $L_{\Phi_2,t^{\gamma}}(\mathbb{R_+})$, into another such Orlicz class, $L_{\Phi_1,t^{\gamma}}(\mathbb{R_+})$. Here, the $\Phi_i, i=1,2$, are nonnegative, nondecreasing functions on $\mathbb{R_+}$, $\gamma \in \mathbb{R}$ and, for any given nonnegative, nondecreasing function $\Phi$ from $\mathbb{R_+}$ onto itself,
\begin{equation*}
L_{\Phi,t^{\gamma}}(\mathbb{R_+}) = \left\lbrace  f \in M(\mathbb{R_+}) : \int_{\mathbb{R_+}} \Phi(k|f(t)|)t^{\gamma}dt < \infty, \  \text{for some} \  k \in \mathbb{R_+} \right\rbrace.
\end{equation*}
One way to measure the size of an $f \in L_{\Phi}(\mathbb{R_+},t^{\gamma})$ is by its gauge
\begin{equation*}
\rho_{\Phi, t^{\gamma}}(f) = \inf \left\lbrace  \lambda>0 : \int_\mathbb{R_+} \Phi \left( \frac{|f(t)|}{\lambda} \right)  \frac{t^{\gamma}}{\lambda}dt \leq 1 \right\rbrace.
\end{equation*}

The fundamental result in this paper, the one on which all others are based, is
\begin{untheorem}\label{gauge norm equivalent modular}\label{main theorem}
Let $T$ be a dilation-commuting operator from $S(\mathbb{R_+})$ to $M(\mathbb{R_+})$. Suppose $\Phi_1$ and $\Phi_2$ are nonnegative, nondecreasing functions from $\mathbb{R_+}$ onto itself and fix $\gamma \in \mathbb{R}, \gamma \neq -1$. Then, there exists $C>0$, independent of $f \in S(\mathbb{R_+})$, such that
\begin{equation}\label{Gauge}
\rho_{\Phi_1,t^{\gamma}}(Tf) \leq C \rho_{\Phi_2,t^{\gamma}}(f) \tag{G}
\end{equation}
if and only if
\begin{equation}\label{Modular}
\int_{\mathbb{R_+}} \Phi_1 \left( |(Tf)(t)| \right)t^{\gamma}dt \leq K \int_{\mathbb{R_+}} \Phi_2 \left( K|f(s)| \right)s^{\gamma}ds, \tag{M}
\end{equation}
in which $K>0$ is independent of $f \in S(\mathbb{R_+})$ \footnote{One easily works out the variant of this in which $\mathbb{R_+}$ is replaced by $\mathbb{R}=(-\infty,\infty)$ and $t^{\gamma}$ by $|t|^{\gamma}$.}.
\end{untheorem}
If $\Phi_1$ and $\Phi_2$ are $s$-convex, that is, 
\begin{equation*}
\Phi_{i} (\alpha x + \beta y) \leq  \alpha^s \Phi_{i}(x)+ \beta^s \Phi_{i}(y), \ \ \ i=1,2,
\end{equation*}
where $0<s<1$ and $\alpha,\beta, x, y \in \mathbb{R_+}$, $\alpha^s + \beta^s=1$, then (\ref{Gauge}) and (\ref{Modular}) are each equivalent to 
\begin{equation}\label{operator continuity}
T : L_{\Phi_2,t^{\gamma}}(\mathbb{R_+}) \rightarrow L_{\Phi_1,t^{\gamma}}(\mathbb{R_+}), \tag{I}
\end{equation}
the injection (\ref{operator continuity}) being continuous with respect to certain metrics defined in terms of $\rho_{\Phi_{1},t^{\gamma}}$ and $\rho_{\Phi_{i},t^{\gamma}}$. See \Cref{operator mapping implies norm inequality} in \Cref{Orlicz classes} below.

The specific dilation-commuting operators we focus on are the Hardy operators
\begin{equation*}
(P_pf)(t)= t^{-\frac{1}{p}} \int_0^t f(s)s^{\frac{1}{p}-1}ds \ \  \text{and} \ \  (Q_qf)(t)= t^{-\frac{1}{q}} \int_t^{\infty} f(s)s^{\frac{1}{q}-1}ds,
\end{equation*}
where 
$p,q \in \mathbb{R_+}$ 
and $f$ is a nonnegative Lebesgue-measurable function on $\mathbb{R}$; the Hardy-Littlewood maximal function
\begin{equation*}
(Mf)(x) = \sup_{\substack{x \in  I \\ I \text{ is an interval}}} \frac{1}{|I|} \int_I
|f(y)|dy, \ f \in M(\mathbb{R}), \ x \in \mathbb{R};
\end{equation*}
the Hilbert transform
\begin{equation*}
(Hf)(x) =  \frac{1}{\pi} \  \text{(P)} \int_{\mathbb{R}} \frac{f(y)}{x-y}dy  = \lim_{\epsilon \rightarrow 0^{+}} \frac{1}{\pi} \int_{|x-y|> \epsilon} \frac{f(y)}{x-y}dy,
\end{equation*}
with $f \in L_1 \left( \frac{1}{1+|y|} \right)$, that is, $\int_{\mathbb{R}} \frac{|f(y)|}{1+ |y|}dy < \infty $, and $x \in \mathbb{R}$.

It is not hard to show that if the inequality (\ref{Gauge}) or (\ref{Modular}) holds for any of these operators when $f \in S(\mathbb{R_+})$, then it holds for that operator when $f$ belongs to its natural domain, that is those $f$ for which the right side of the inequality is finite.

The above operators will be treated in \Cref{operators Pp and Qq}, \Cref{Maximal function} and \Cref{The Hilbert transform}, respectively, following the proof of \Cref{main theorem} in \Cref{proof of main theorem}. Background on gauges such as $\rho_{\Phi,t^{\gamma}}$ is given in \Cref{Orlicz classes}.

\section{Orlicz classes}\label{Orlicz classes}
Let $(X,M,\mu)$ be a totally $\sigma$-finite measure space and denote by $M(X)$ the set of $\mu$-measurable functions from $X$ to the real line $\mathbb{R}$. Given a nondecreasing function $\Phi$ from $\mathbb{R_+}$ onto itself its corresponding Orlicz class is
\begin{equation*}
L_{\Phi,\mu}(X) = \left\lbrace f \in M(X): \int_X \Phi (k|f(x)|)d\mu(x)< \infty, \ \  \text{for some } \ k \in \mathbb{R_+} \right\rbrace.
\end{equation*}
The functional $\rho_{\Phi,\mu}$ defined at $f \in M(X)$ by
\begin{equation*}
\rho_{\Phi,\mu}(f) = \inf \left\lbrace \lambda>0: \int_X \Phi \left( \frac{|f(x)|}{\lambda} \right) \frac{d\mu(x)}{\lambda} \leq 1  \right\rbrace
\end{equation*}
is finite if and only if $f \in L_{\Phi,\mu}(X)$.

This functional has the following properties
\begin{enumerate}
\item $\rho_{\Phi,\mu}(f)= \rho_{\Phi,\mu}(|f|) \geq 0$, with $\rho_{\Phi,\mu}(f)=0$ if and only if $f=0$ $\mu$-a.e.;
\item $\rho_{\Phi,\mu}(cf)$ is a nondecreasing function of $c$ from $\mathbb{R_+}$ onto itself if $f \neq 0$ $\mu$-a.e.;
\item $\rho_{\Phi,\mu}(f+g) \leq \rho_{\Phi,\mu}(f) + \rho_{\Phi,\mu}(g)$;
\item $0 \leq f_n \uparrow f$ implies $\rho_{\Phi,\mu}(f_n) \uparrow \rho_{\Phi,\mu}(f)$;
\item $\rho_{\Phi,\mu}(\chi_{E})<\infty$ for all $E \subset X$ such that $\mu(E)<\infty$.
\end{enumerate}
See [M], for example.

The functional $\rho_{\Phi,\mu}$ is a so-called $F$-norm on the linear space $L_{\Phi,\mu}(X)$ that makes it into a complete linear topological space under the metric
\begin{equation*}
d_{\Phi,\mu}(f,g)=\rho_{\Phi,\mu}(f-g).
\end{equation*}
If $\Phi$ is $s$-convex for some $s \in (0,1]$ then,
\begin{equation*}
\rho_{\Phi,\mu}( \alpha f+ \beta g) \leq \alpha^s \rho_{\Phi,\mu}(f) + \beta^s \rho_{\Phi,\mu}(g),
\end{equation*}
when $f, g \in L_{\Phi,\mu}(X)$. Moreover, the functional
\begin{equation*}
\rho_{\Phi,\mu}^{(s)}(f) = \inf \left\lbrace \lambda>0: \int_X \Phi \left( \frac{|f(x)|}{\lambda^{1/s}} \right) d\mu(x) \leq 1  \right\rbrace
\end{equation*}
satisfies
\begin{equation*}
\rho_{\Phi,\mu}^{(s)}(cf)= c^s \rho_{\Phi,\mu}^{(s)}(f), \ \ \ c\geq 0, 
\end{equation*}
as well as properties $1-5$ above, so, in particular, $\rho_{\Phi,\mu}^{(1)}(f)$ is a norm. Finally, the metric
\begin{equation*}
d_{\Phi,\mu}^{(s)}(f,g)= d_{\Phi,\mu}^{(s)}(f-g)
\end{equation*}
is equivalent to the metric $d_{\Phi,\mu}(f,g)$.

\begin{proposition}\label{operator mapping implies norm inequality}
Let $(X,M,\mu)$ and $(Y,N,\nu)$ be totally $\sigma$-finite measure spaces. Suppose $\Phi_1$ and $\Phi_2$  are $s$-convex nondecreasing functions from $\mathbb{R_+}$ onto itself. Then, the linear operator $T$ mapping $L_{\Phi_2,\nu}(Y)$ continuously into $L_{\Phi_1,\mu}(X)$ (with respect to the metrics defined by $\rho_{\Phi_2,\nu}^{(s)}$ and $\rho_{\Phi_1,\mu}^{(s)}$) implies
\begin{equation}
\rho_{\Phi_1,\mu}^{(s)}(Tf) \leq C \rho_{\Phi_2,\nu}^{(s)}(f) \tag{$G^{(s)}$},
\end{equation}
with $C>0$ independent of $f \in L_{\Phi_2,\nu}(Y)$.
\end{proposition}
\begin{proof}
Fix $f_0 \in L_{\Phi_2,\nu}(Y)$. Since $T$ is continuous at $f_0$, there is a $\delta>0$ such that
\begin{equation*}
\rho_{\Phi_1,\mu}^{(s)}(Tf-Tf_{0}) < 1
\end{equation*}
for all $f \in L_{\Phi_2,\nu}(Y)$ satisfying $\rho_{\Phi_2,\nu}^{(s)}(f-f_0)< \delta$. Given $f \in L_{\Phi_2,\nu}(Y)$, set $g= \frac{\eta^{1/s}}{\rho_{\Phi_2,\nu}^{(s)}(f)^{1/s}}f$, for a fixed $\eta, 0 < \eta < \delta$. Then,
\begin{equation*}
\frac{\eta^{1/s}}{\rho_{\Phi_2,\nu}^{(s)}(f)^{1/s}} Tf = Tg = T(g+f_0)-Tf_0
\end{equation*}
and
\begin{equation*}
  \rho^{(s)}_{\Phi_1,\mu} \left( \frac{\eta^{1/s}}{\rho_{\Phi_2,\nu}^{(s)}(f)^{1/s}} Tf \right) =  \rho^{(s)}_{\Phi_1,\mu} \left(  T(g+f_0)-Tf_0 \right) < 1,
\end{equation*}
since
\begin{equation*}
 \rho_{\Phi_2,\nu}^{(s)}\left( g+f_0 - f_0 \right)=  \rho_{\Phi_2,\nu}^{(s)}\left( g \right) \leq \eta < \delta.
\end{equation*}
Indeed,
\begin{equation*}
\int_Y \Phi_2 \left( \frac{g}{\eta
^{1/s}} \right) d\nu = \int_Y \Phi_2 \left( \frac{f}{\rho_{\Phi_2,\nu}^{(s)}(f)^{1/s}} \right)d\nu \leq 1.
\end{equation*}
Now,
\begin{equation*}
\rho^{(s)}_{\Phi_1,\mu} \left( \frac{\eta^{1/s}}{\rho_{\Phi_2,\nu}^{(s)}(f)^{1/s}} Tf \right) < 1,
\end{equation*}
implies,
\begin{equation*}
\int_X \Phi_1 \left( \frac{Tf}{ \rho_{\Phi_2,\nu}^{(s)}(f)^{1/s}  /   \eta^{1/s}    } \right) d\mu \leq 1,
\end{equation*}
which, in turn, means that
\begin{equation*}
\rho_{\Phi_1,\mu}^{(s)}(Tf) \leq \eta^{-1} \rho_{\Phi_2,\nu}^{(s)}(f).
\end{equation*}
\end{proof}

\section{\texorpdfstring{Proof of \cref{gauge norm equivalent modular}}{}}\label{proof of main theorem}
We will require the connection between a modular inequality like (\ref{Modular}) and certain gauge inequalities, (\ref{General weighted gauge}), given in some generality in
\begin{proposition}\label{Weighted modular equivalent to family of gauge}
Let $t,u,v$ and $w$ be positive measurable functions, called weights, on $\mathbb{R_+}$. Suppose $\Phi_1$ and $\Phi_2$ are nonnegative, nondecreasing functions from $\mathbb{R_+}$ onto itself. Given $\epsilon > 0$, define the weighted gauges of $f,g \in M(\mathbb{R_+})$ by
\begin{equation*}
\rho_{\Phi_2, u, \epsilon v}(f)   = \inf \left\lbrace  \lambda > 0 : \int_{\mathbb{R_+}} \Phi_2 \left( \frac{u(y)|f(y)|}{\lambda} \right)\frac{\epsilon}{\lambda} v(y)dy \leq 1   \right\rbrace 
\end{equation*}
 and
\begin{equation*}
\rho_{\Phi_1, t, \epsilon w}(g)   = \inf \left\lbrace  \lambda > 0 : \int_{\mathbb{R_+}} \Phi_2 \left( \frac{t(x)|g(x)|}{\lambda} \right)\frac{\epsilon}{\lambda} w(x)dx \leq 1   \right\rbrace .
\end{equation*}
Then, a positively homogeneous operator $T$ from $S(\mathbb{R_+})$ to $M(\mathbb{R_+})$ satisfies
\begin{equation}\label{General weighted modular}
\int_{\mathbb{R_+}} \Phi_1 \left( t(x) |(Tf)(x)| \right)w(x)dx \leq K\int_{\mathbb{R_+}} \Phi_2 \left( K u(y)|f(y)| \right)v(y)dy, \tag{M}
\end{equation}
if and only if it satisfies the modular inequalities
\begin{equation}\label{General weighted gauge}
\rho_{\Phi_1,t, \epsilon w}(Tf) \leq C \rho_{\Phi_2,u, \epsilon v}(f), \tag{ $G_{\epsilon}$ }
\end{equation}
in which $K>0$ is independent of  $f \in S(\mathbb{R_+})$  and $C>0$ is independent of both $f \in S(\mathbb{R_+})$ \textbf{and} $ \epsilon > 0$.
\end{proposition}
\begin{proof}
Suppose (\ref{General weighted gauge}) holds. Fix $f \in S(\mathbb{R_+})$, $f \not\equiv 0$, and put
\begin{equation*}
\epsilon = \left( \int_\mathbb{R_+} \Phi_2 \left(  u(y)|f(y)| \right)v(y)dy \right)^{-1}.
\end{equation*}
Then,
\begin{equation*}
\int_\mathbb{R_+} \Phi_2 \left(  u(y)|f(y)| \right) \epsilon v(y)dy = 1,
\end{equation*}
so
\begin{equation*}
\rho_{\Phi_2,u, \epsilon v} (f) \leq 1,
\end{equation*}
whence (\ref{General weighted gauge}) implies
\begin{equation*}
\rho_{\Phi_1,t, \epsilon w}(Tf) \leq C.
\end{equation*}
Thus,
\begin{equation*}
\int_\mathbb{R_+} \Phi_1 \left( \frac{t(x) |(Tf)(x)|}{C} \right)\frac{w(x)}{C}dx \leq  \frac{1}{\epsilon} = \int_\mathbb{R_+} \Phi_2 \left(  u(y)|f(y)| \right)v(y)dy.
\end{equation*}
Replacing $f$ by $Cf$ and using the fact that $T$ is positively homogeneous yields (\ref{General weighted modular}), with $K=C$.

For the converse, fix $f \in S(\mathbb{R_+})$ and $\epsilon>0$. Let $\alpha = \rho_{\Phi_2,u, \epsilon v}(f)$, so that
\begin{equation*}
\int_\mathbb{R_+} \Phi_2 \left(    \frac{u(y)|f(y)|}{\alpha}  \right) \frac{\epsilon}{\alpha} v(y)dy \leq 1.
\end{equation*}
By (\ref{General weighted modular}), then,
\begin{equation*}
\begin{split}
\int_\mathbb{R_+} \Phi_1 \left(    \frac{t(x)|(Tf)(x)|}{K \alpha}  \right) \frac{\epsilon}{K\alpha} w(x)dx & =  \epsilon \int_\mathbb{R_+} \Phi_1 \left(    \frac{t(x)|(Tf)(x)|}{K \alpha}  \right)  \frac{w(x)}{K\alpha}dx \\
&  \leq \int_\mathbb{R_+} \Phi_2 \left(    \frac{u(y)|f(y)|}{\alpha}  \right) \frac{\epsilon}{\alpha} v(y)dy \\
& \leq 1,
\end{split}
\end{equation*}
which amounts to
\begin{equation*}
\rho_{\Phi_1,t, \epsilon w}(Tf) \leq K \alpha = C \rho_{\Phi_2,u, \epsilon v}(f),
\end{equation*}
with $C=K>0$ independent of $f \in S(\mathbb{R_+})$ and $\epsilon>0$.
\end{proof}

\begin{proof}[Proof of \cref{gauge norm equivalent modular}]
According to \Cref{Weighted modular equivalent to family of gauge}, the modular inequality (\ref{Modular}) is equivalent to the family of uniform gauge inequalities
\begin{equation}\label{Weighted gauge}
\rho_{\Phi_1,\epsilon t^{\gamma}}(Tf) \leq C \rho_{\Phi_2, \epsilon t^{\gamma}}(f) \tag{$G_{\epsilon}^{\gamma}$}
\end{equation}
with $C>0$  independent of both $f$ and $\epsilon > 0$.

In particular, ($ \textcolor{red}{G_{1}^{\gamma}}$) is (\ref{Gauge}), so (\ref{Modular}) implies (\ref{Gauge}).

Next, we prove (\ref{Gauge}) implies (\ref{Weighted gauge}), which amounts to showing
\begin{equation*}
\int_\mathbb{R_+} \Phi_1 \left( \frac{|(Tf)(t)|}{  C \rho_{\Phi_2, \epsilon s^{\gamma}}(f)   } \right) \frac{\epsilon t^{\gamma}}{ C \rho_{\Phi_2, \epsilon s^{\gamma}}(f) }  dt \leq 1.
\end{equation*}
Letting $z= \epsilon^{\delta}t$, $\delta = \frac{1}{1+ \gamma}$, the latter reads
\begin{equation*}
\int_\mathbb{R_+} \Phi_1 \left( \frac{|(Tf)(z/ \epsilon^{\delta})|}{  C \rho_{\Phi_2, \epsilon s^{\gamma}}(f)   } \right) \frac{z^{\gamma}} { C \rho_{\Phi_2, \epsilon s^{\gamma}}(f)  } dz \leq 1,
\end{equation*}
or, since $T$ commutes with dilations,
\begin{equation*}
\int_\mathbb{R_+} \Phi_1 \left( \frac{|T(f( \frac{1}{\epsilon^{\delta}} \cdot ))(z)|}{  C \rho_{\Phi_2, \epsilon s^{\gamma}}(f)   } \right) \frac{z^{\gamma}} { C \rho_{\Phi_2, \epsilon s^{\gamma}}(f)  } dz \leq 1.
\end{equation*}
But,
\begin{equation*}
\begin{split}
\rho_{\Phi_2, \epsilon s^{\gamma}}(f) &  = \inf \left\lbrace  \lambda > 0 : \int_\mathbb{R_+} \Phi_2 \left( \frac{|f(s)|}{\lambda} \right) \frac{\epsilon}{\lambda} s^{\gamma}ds \leq 1   \right\rbrace \\
& = \inf \left\lbrace  \lambda > 0 : \int_\mathbb{R_+}\Phi_2 \left( \frac{|{ \textstyle{\left( f \left( \frac{1}{\epsilon^{\delta}} y \right) \right)}}|}{\lambda} \right) \frac{y^{\gamma}} {\lambda}dy \leq 1   \right\rbrace\\
& = \rho_{\Phi_2, t^{\gamma}} { \textstyle{\left( f \left( \frac{1}{\epsilon^{\delta}} \cdot \right) \right)}},
\end{split}
\end{equation*}
where in the first equality, we have made the change of variable $s= y / \epsilon^{\delta}$. Altogether, then, (\ref{Weighted gauge}) is the same as (\ref{Gauge}), with $f$ replaced by $f \left( \frac{1}{\epsilon^{\delta}} \cdot \right) $.
\end{proof}

\section{ \texorpdfstring {The operators $P_p$ and $Q_q$}{}}\label{operators Pp and Qq}
We will sometimes need to work with  nonnegative, nondecreasing $\Phi$ on $\mathbb{R_+}$ that are Young functions, by which is meant
\begin{equation*}
\Phi(t)= \int_0^t \phi(s)ds, \  t \in \mathbb{R_+},
\end{equation*}
where $\phi$ is nondecreasing on $\mathbb{R_+}$, with $\phi(0^{+})=0$ and $\lim_{t \rightarrow \infty} \phi(s)=\infty$. The Young function, $\Psi$, complementary to such a $\Phi$ is defined by
\begin{equation*}
\Psi(t)= \int_0^t \phi^{-1}(s)ds, \  t \in \mathbb{R_+}.
\end{equation*}
%
%
%

\begin{untheorem}\label{characterization of $P_p$ and $Q_q$}
Fix $p, \gamma \in \mathbb{R}, \gamma \neq -1, \gamma +1 - {\textstyle\frac{1}{p}}  \neq 0$. Let $P_p$ be defined as in the introduction.
Suppose that $\Phi_1$ and $\Phi_2$ are nonnegative, nondecreasing functions from $\mathbb{R_+}$ onto itself and that, for $i=1,2$, $\rho_{\Phi_i,t^{\gamma}}$ are the corresponding F-norms on $L_{\Phi_i, t^{\gamma}}(\mathbb{R_+})$.

Then, the following are equivalent:
\begin{equation}\label{(1a)}
\rho_{\Phi_1,t^{\gamma}}(P_pf) \leq L \, \rho_{\Phi_2,t^{\gamma}}(f), \tag{G}
\end{equation}
$L>0$ being independent of $0 \leq f \in  M(\mathbb{R_+})$;
\begin{equation}\label{(2a)}
\int_{\mathbb{R_+}} \Phi_1 \left( (P_pf)(t) \right)t^{\gamma}dt \leq K \int_{\mathbb{R_+}} \Phi_2 \left( K f(s) \right)s^{\gamma}ds,  \tag{M}
\end{equation}
in which $K>0$ is independent of $0 \leq f \in  M(\mathbb{R_+})$.

Further, if $\Phi_1$ and $\Phi_2$ are $s$-convex for some $s$, $0 < s \leq 1$, then  (\ref{(1a)}) and (\ref{(2a)}) are equivalent to 
\begin{equation}\label{operator Pp continuity}
P_p : L_{\Phi_2,t^{\gamma}}(\mathbb{R_+}) \rightarrow L_{\Phi_1,t^{\gamma}}(\mathbb{R_+}), \tag{I}
\end{equation}
the injection (\ref{operator Pp continuity}) being continuous with respect to the metrics defined in terms of the $s$-convex F-norms $\rho_{\Phi_1}^{(s)}$ and $\rho_{\Phi_2}^{(s)}$.

Finally, for a general nondecreasing function $\Phi_1$ and a Young function $\Phi_2$, with complementary function $\Psi_2$, (\ref{(1a)}) and (\ref{(2a)}) are equivalent to
\begin{equation}\label{(3a)}
\int_0^t \Psi_2 \left( \frac{\alpha(t)}{C s^{1- \frac{1}{p}+ \gamma}} \right) s^{\gamma}ds \leq \alpha(t)= \int_t^{\infty} \Phi_1( s^{- \frac{1}{p}}) s^{\gamma}ds < \infty, \tag{BK}
\end{equation}
or
\begin{equation*}
\int_0^t \phi_2^{-1} \left( \frac{\alpha(t)}{C s^{1- \frac{1}{p}+ \gamma}} \right) s^{\frac{1}{p}-1}ds \leq C, \ \ t \in \mathbb{R_+}.
\end{equation*}
\end{untheorem}
\begin{proof}[Proof of \cref{characterization of $P_p$ and $Q_q$}]

Since $P_p$ commutes with dilations, $(\ref{(1a)})$ and $(\ref{(2a)})$ are equivalent, in view of \cref{gauge norm equivalent modular}. In case $\Phi_1$ and $\Phi_2$ are $s$-convex, \Cref{operator mapping implies norm inequality} ensures that $(\ref{(1a)})$, and hence $(\ref{(2a)})$, is equivalent when $\Phi_2$ is a Young function with complementary function $\Psi_2$.

%

The inequality in $(\ref{(2a)})$ reads
\begin{equation*}
\int_{\mathbb{R_+}} \Phi_1 \left( t^{-\frac{1}{p}} \int_0^{t} f(s) s^{\frac{1}{p}-1}ds \right)t^{\gamma}dt \leq \int_{\mathbb{R_+}} \Phi_2 \left( Kf(s) \right)s^{\gamma}ds, \ 0 \leq f \in M(\mathbb{R_+}).
\end{equation*}
Replacing $f(s)s^{\frac{1}{p}-1}$ by $g(s)$, we have
\begin{equation*}
\int_{\mathbb{R_+}} \Phi_1 \left( t^{-\frac{1}{p}} \int_0^{t} g(s)ds \right)t^{\gamma}dt \leq \int_{\mathbb{R_+}} \Phi_2 \left( Kg(s)s^{1-\frac{1}{p}} \right)s^{\gamma}ds, \ 0 \leq f \in M(\mathbb{R_+}).
\end{equation*}
According to \Cref{conditions for weak vs weighted weak modular} (in Appendix I), this latter holds if and only if
\begin{equation*}
\int_0^t \Psi_2 \left( \frac{\alpha(\lambda, t)}{C \lambda y^{1 - \frac{1}{p}+ \gamma}} \right) y^{\gamma} dy \leq \alpha(\lambda, t)= \int_t^{\infty} \Phi_1 \left( \lambda z^{- \frac{1}{p}} \right)z^{\gamma}dz < \infty,
\end{equation*}
with constant $C>0$ independent of $\lambda, t \in \mathbb{R_+}$. Letting $y= \lambda^p s$ and $z= \lambda^p s$ in the above integrals we obtain
\begin{equation*}
\int_0^{\lambda^{-p} t} \Psi_2 \left( \frac{\alpha(\lambda^{-p} t)}{C s^{1 - \frac{1}{p}+ \gamma}} \right) s^{\gamma} ds \leq \alpha(\lambda^{-p} t)= \int_{\lambda^{-p} t}^{\infty} \Phi_1 \left( s^{- \frac{1}{p}} \right)s^{\gamma}ds < \infty,
\end{equation*}
Replacing $\lambda^{-p} t$ by $t$ yields $(\ref{(3a)})$.
\end{proof}

\begin{remark}
We put the (second) condition in $(\ref{(3a)})$ above in a form found useful, in [KP], for certain applications. As a by product we get more precise connections between the indices $p$ and $\gamma$.
{ \renewcommand\labelenumi{(\theenumi)}
\begin{enumerate}
\item $1- \frac{1}{p}+ \gamma =0$. The condition reads
\begin{equation*}
p \phi_2^{-1} \left( { \textstyle\frac{\alpha(t)}{C}  } \right) \leq C t^{-\frac{1}{p}}.
\end{equation*}
\item $1- \frac{1}{p}+ \gamma \neq 0$. We set $y = {\textstyle\frac{\alpha(t)}{s^{1- \frac{1}{p}+ \gamma}}}$ in the integral on the left side of the condition to get, with $\lambda(t)= {\textstyle\frac{\alpha(t)}{t^{1- \frac{1}{p}+ \gamma}}}$
\begin{equation}\label{first integral}
\int_{\lambda(t)}^{\infty} \phi_2^{-1} \left( {\textstyle\frac{y}{C}} \right) \frac{dy}{y^{      \frac{\gamma +1 }{ 1- \frac{1}{p}+ \gamma}         }}   \leq \left( 1- {\textstyle\frac{1}{p}} + \gamma \right) \alpha(t)^{\frac{1}{ (1+ \gamma)p-1}},
\end{equation}
when $1- \frac{1}{p}+ \gamma >0$, and
\begin{equation}\label{second integral}
\int_0^{\lambda(t)} \phi_2^{-1} \left( {\textstyle\frac{y}{C}} \right) \frac{dy}{y^{      \frac{\gamma +1 }{ 1- \frac{1}{p}+ \gamma}         }}   \leq - \left( 1- {\textstyle\frac{1}{p}} + \gamma \right) \alpha(t)^{\frac{1}{ (1+ \gamma)p-1}},
\end{equation}
when $1- \frac{1}{p}+ \gamma <0$.

Observe that for the integral in (\ref{first integral}) to make sense we require $\gamma+1>0$ or $\gamma>-1$.

Again, the change of variable $y=s^{-\frac{1}{p}}$ in the integral giving $\alpha(t)$ yields
\begin{equation}\label{third integral}
\alpha(t)= p \int_0^{t^{-\frac{1}{p}}} \frac{\Phi_1(y)}{y}  \frac{dy}{ y^{(\gamma+1)p}}, \ \ \ \text{when} \ \ \ p>0,
\end{equation}
and
\begin{equation}\label{fourth integral}
\alpha(t)= -p \int_{t^{-\frac{1}{p}}}^{\infty} \frac{\Phi_1(y)}{y}  \frac{dy}{ y^{(\gamma+1)p}}, \ \ \ \text{when} \ \ \ p<0,
\end{equation}

In (\ref{fourth integral}) we need $\gamma+1<0$ or $\gamma<-1$.

Altogether, then, $\ref{(3a)}$ amounts to (\ref{first integral}) with $\alpha(t)$ given by (\ref{third integral}), when $p>0$ and $\gamma > -1 + \frac{1}{p}$ and to (\ref{second integral}) with $\alpha(t)$ given by (\ref{fourth integral}) when $p<0$ and $\gamma < -1 + \frac{1}{p}$.


\end{enumerate}
}
\end{remark}

\begin{corollary}
Fix $q, \gamma \in \mathbb{R}, \gamma \neq -1$. Let $Q_q$ be defined as in the introduction. Suppose that $\Phi_1$ and $\Phi_2$ are Young functions having complementary functions $\Psi_1$ and $\Psi_2$, respectively. Then, the following are equivalent: 
\begin{equation}\label{(4)}
\rho_{\Phi_1,t^{\gamma}}^{(1)}(Q_qf) \leq L \, \rho_{\Phi_2,t^{\gamma}}^{(1)}(f), \tag{G}
\end{equation}
$L>0$ being independent of $0 \leq f \in  M(\mathbb{R_+})$;
\begin{equation}\label{(5)}
\int_{\mathbb{R_+}} \Phi_1 \left( |(Q_qf)(t)| \right)t^{\gamma}dt \leq \int_{\mathbb{R_+}} \Phi_2 \left( K|f(s)| \right)s^{\gamma}ds, \tag{M}
\end{equation}
in which $K>0$ is independent of $0 \leq f \in  M(\mathbb{R_+})$;
\begin{equation}\label{operator Qq continuity}
Q_q : L_{\Phi_2,t^{\gamma}}(\mathbb{R_+}) \rightarrow L_{\Phi_1,t^{\gamma}}(\mathbb{R_+}), \tag{I}
\end{equation}
the injection (\ref{operator Qq continuity}) being continuous with respect to the norms $\rho_{\Phi_1}^{(1)}$ and $\rho_{\Phi_2}^{(1)}$;
\begin{equation}\label{(6)}
\int_0^t \Phi_1 \left( \frac{\beta(t)}{C s^{ \frac{1}{q}}} \right) s^{\gamma}ds \leq \beta(t)= \int_t^{\infty} \Psi_2( s^{\frac{1}{q}-1- \gamma}) s^{\gamma}ds < \infty, \tag{BK}
\end{equation}
or 
\begin{equation*}
\int_0^t \phi_1 \left( \frac{\beta(t)}{C s^{\frac{1}{q}}} \right) s^{\gamma - \frac{1}{q}}ds \leq C, \ \ t \in \mathbb{R_+}.
\end{equation*}
\end{corollary}
\begin{proof}
If, for $i=1,2$, $\Phi_i$ is a Young function, then $\rho_{\Phi_i, t^{\gamma}}^{(1)}$ is a Banach function norm with its associate Banach function norm equivalent to $\rho_{\Psi_i, t^{\gamma}}^{(1)}$; see \cite[p. 274-275]{BS}. In particular, The Principle of Duality for such norms ensures that  
\begin{equation*}
\rho_{\Phi_1,t^{\gamma}}^{(1)}(Q_qf) \leq K \rho_{\Phi_2,t^{\gamma}}^{(1)}(f),\ 0 \leq f \in M(\mathbb{R_+}),
\end{equation*}
if and only if for all $ 0 \leq f,g \in M(\mathbb{R_+})$, $\rho_{\Phi_2,s^{\gamma}}^{(1)}(f), \rho_{\Psi_1,t^{\gamma}}^{(1)}(g) \leq 1$, one has
\begin{equation*}
\int_{\mathbb{R_+}} g(t)t^{-\frac{1}{q}} \int_{t}^{\infty} f(s)s^{\frac{1}{q}-1}ds \ t^{\gamma}dt < \infty.
\end{equation*}
But, this means
\begin{align*}
\infty > \int_{\mathbb{R_+}} g(t)t^{-\frac{1}{q}} \int_{t}^{\infty} f(s)s^{\frac{1}{q}-1}ds \ t^{\gamma}dt &= \int_{\mathbb{R_+}} f(s)s^{\frac{1}{q}-1} \int_0^s g(t)t^{-\frac{1}{q}+ \gamma} dt ds,
\end{align*}
for all $0 \leq f,g \in M(\mathbb{R_+})$, with $ \rho_{\Psi_1,t^{\gamma}}^{(1)}(g),\rho_{\Phi_2,s^{\gamma}}^{(1)}(f) \leq 1$, that is
\begin{equation*}
\rho_{\Psi_{2},t^{\gamma}}^{(1)}(P_{r}g) \leq K \rho_{\Psi_{1},t^{\gamma}}^{(1)}(g), \ \ 0 \leq g \in M(\mathbb{R_+}),
\end{equation*}
where $\frac{1}{r}= 1- \frac{1}{q}+ \gamma$.

The  argument for $P_p$ in \cref{characterization of $P_p$ and $Q_q$} shows this last inequality holds if and only if the condition in $(\ref{(6)})$ does.

Moreover, since $Q_q$ is a dilation commuting linear operator the same considerations as in the proof of \cref{characterization of $P_p$ and $Q_q$}, shows $(\ref{(4)})$, $(\ref{(5)})$ and $(\ref{operator Qq continuity})$ are equivalent.

This completes the proof.
\end{proof}

\section{ \texorpdfstring {The Hardy-Littlewood maximal operator $M$}{}}\label{Maximal function}

\begin{untheorem}\label{characterization for M}
Fix $\gamma >-1$. Let $\Phi(t)=\int_0^t \phi(s)ds$ be a Young function with complementary function $\Psi$. Then, the following assertions are equivalent:

\begin{equation}\label{(M1)}
\rho_{\Phi,|x|^{\gamma}}^{(1)}(Mf) \leq L \rho_{\Phi, |x|^{\gamma}}^{(1)}(f), \tag{G}
\end{equation}
where $L>0$, is independent of $f \in M(\mathbb{R})$;
\begin{equation}\label{operator continuity of M}
M : L_{\Phi, |x|^{\gamma}}(\mathbb{R}) \rightarrow L_{\Phi, |x|^{\gamma}}(\mathbb{R}); \tag{I}
\end{equation}
\begin{equation}\label{(M2)}
\int_{\mathbb{R}} \Phi \left( (Mf)(x) \right) |x|^{\gamma}dx \leq \int_{\mathbb{R}} \Phi(K |f(y)|)|y|^{\gamma}dy, \tag{M}
\end{equation}
in which $K>0$ doesn't depend on $f \in M(\mathbb{R})$;
\begin{equation}\label{(M3)}
\Psi(2t) \leq C \Psi(t) \tag{BK}
\end{equation}
and either $-1 < \gamma < 0$ or
\begin{equation}\label{our condition for maximal function}
\gamma \geq 0 \ \ \ \text{and} \ \ \   \frac{1}{t} \int_0^t \phi^{-1}(s^{-\gamma})ds \leq C \phi^{-1}(Ct^{-\gamma}),
\end{equation}
for some $C \geq 1$ independent of $t \in \mathbb{R_+}$.
\end{untheorem}
\begin{proof}[Proof of \cref{characterization for M}]
In view of \cref{gauge norm equivalent modular}, $\ref{(M1)}$, $\ref{operator continuity of M}$ and $\ref{(M2)}$ are equivalent . Again, by a special case of Theorem 1 in \cite{BK1}, $\ref{(M2)}$ holds if and only if
\begin{equation*}
\Psi(2t) \leq C \Psi(t), \ \ \ t \in \mathbb{R_+}, 
\end{equation*}
and
\begin{equation}\label{BK1 condition}
\frac{1}{\mu_{\gamma}(I)} \int_I \Psi \left( \frac{1}{C}   \frac{\Phi(\lambda)}{\lambda} \frac{\mu_{\gamma}(I) }{|I||x|^{\gamma}}  \right) |x|^{\gamma} dx \leq \Phi(\lambda)
\end{equation}
with $C \geq 1$ independent of $I \subset \mathbb{R}$ and $\lambda \in  \mathbb{R_+}$; here 
\begin{equation*}
\mu_{\gamma}(I)= \int_I |x|^{\gamma}dx .
\end{equation*}
Since
\begin{equation*}
\textstyle\frac{t}{2} \phi^{-1}({\textstyle\frac{t}{2}}) \leq \Psi(t) \leq t \phi^{-1}(t), \ t \in \mathbb{R_+},
\end{equation*}
(\ref{BK1 condition}) is equivalent to
\begin{equation*}
\frac{1}{|I|} \int_I \phi^{-1} \left( \frac{1}{C}   \phi(\lambda) \frac{\mu_{\gamma}(I) }{|I|}   \frac{1}{|x|^{\gamma}}\right) dx \leq C \lambda,
\end{equation*}
in which $C \geq 1$ does not depend on $I \subset \mathbb{R}$ or $\lambda \in  \mathbb{R_+}$.

We observe that the assumption $\gamma>-1$ was necessary in order that $\mu_{\gamma}(I)$ be finite for all intervals $I \subset \mathbb{R_+}$. Further, one readily shows that for $I=[a,b],$
\begin{equation*}
\frac{\mu_{\gamma}(I)}{|I|} \approx \max [|a|,|b|]^{\gamma}.
\end{equation*}
Assume, first that $ab \geq 0$, say $0 \leq a < b$. Then (\ref{BK1 condition}) amounts to
\begin{equation*}
\frac{1}{b-a} \int_a^b \phi^{-1} \left(    \phi(\lambda) \frac{1}{C}  \left( \frac{b}{x} \right)^{\gamma}\right)dx \leq C \lambda.
\end{equation*}
when $-1 < \gamma < 0$, which automatically holds with $C=1$, since $\frac{1}{C}  \left( \frac{b}{x} \right)^{\gamma} \leq 1$. The same is true when $\gamma \geq 0$ and $a > \frac{b}{2}$ with $C= 2^{\gamma}$.

Suppose, then, $\gamma \geq 0$ and $0 \leq a \leq \frac{b}{2}$. It suffices to take $a=0$, so that we are looking at
\begin{equation*}
\frac{1}{b} \int_0^b \phi^{-1} \left(    \phi(\lambda) \frac{1}{C}  \left( \frac{b}{x} \right)^{\gamma}\right)dx \leq C \lambda,
\end{equation*}
or, setting $x=by$, at
\begin{equation*}
 \int_0^1 \phi^{-1} \left(    \phi(\lambda) \frac{1}{C}  y^{-\gamma}\right)dy \leq C \lambda.
\end{equation*}
Let $s = \phi(\lambda)^{-\frac{1}{\gamma}}y$ to get
\begin{equation*}
 \int_0^{\phi(\lambda)^{-\frac{1}{\gamma}}} \phi^{-1} \left( \frac{s^{-\gamma}}{C}  \right)  \phi(\lambda)^{\gamma}ds \leq C \lambda=C \phi^{-1} \left( \left( \phi(\lambda)^{-\frac{1}{\gamma}} \right)^{-\gamma} \right).
\end{equation*}
Taking $t= \phi(\lambda)^{-\frac{1}{\gamma}}$, we arrive at (\ref{our condition for maximal function}).

Finally, consider $\gamma \geq 0$ and $a < 0 < b$. In that case, (\ref{BK1 condition}) reduces to
\begin{equation*}
\frac{1}{d} \int_0^d \phi^{-1} \left(    \phi(\lambda) \frac{1}{C}  \left( \frac{d}{x} \right)^{\gamma}\right)dx \leq C \lambda,
\end{equation*}
$d= \max[|a|,b]$, namely, to the previous case.
\end{proof}

\section{ \texorpdfstring {The Hilbert transform $H$}{}}\label{The Hilbert transform}
\begin{untheorem}\label{characterization for H}
Fix $\gamma >-1$. Let $\Phi(t)=\int_0^t \phi(s)ds$ be a Young function with complementary function $\Psi$. Then, the following assertions concerning the Hilbert transform, $H$, are equivalent:
\begin{equation}\label{(H1)}
\rho_{\Phi,|x|^{\gamma}}^{(1)}(Hf) \leq K \rho_{\Phi, |x|^{\gamma}}^{(1)}(f), \tag{G}
\end{equation}
where $K>0$, is independent of $f \in L_{\Phi} \left( |x|^{\gamma} \right)$;
\begin{equation}\label{continuity of Hilbert transform H}
H : L_{\Phi, |x|^{\gamma}}(\mathbb{R}) \rightarrow L_{\Phi, |x|^{\gamma}}(\mathbb{R}); \tag{I}
\end{equation} 
\begin{equation}\label{(H2)} 
\int_{\mathbb{R}} \Phi \left( |(Hf)(x)| \right) |x|^{\gamma}dx \leq \int_{\mathbb{R}} \Phi(K |f(y)|)|y|^{\gamma}dy < \infty, \tag{M}
\end{equation}
in which $K>0$ doesn't depend on $f \in L_{1} \left( \frac{1}{1+|x|} \right)$;\\
\emph{(BK)} There exists $C>0$ such that, for all $t \in \mathbb{R_+}$,
\begin{equation}\label{Delta two for H}
\Phi(2t) \leq C\Phi(t),
\end{equation}
\begin{equation}\label{Delta two complement for H}
\Psi(2t) \leq C \Psi(t)
\end{equation}
and
\begin{equation}\label{our condition for Hilbert transform}
\frac{1}{t} \int_0^t \phi^{-1}(s^{-\gamma})ds \leq  \phi^{-1}(Ct^{-\gamma}).
\end{equation}
\end{untheorem}
The condition (\ref{our condition for Hilbert transform}) clearly holds if $\gamma \leq 0$. As for $\gamma >0$, the following lemma shows (\ref{our condition for Hilbert transform}) amounts to the condition $A_{\phi}$ in \cite{KT}, provided one has  (\ref{Delta two for H}).

\begin{lemma}
Fix $\gamma >0$ and let $\Phi(t)=\int_0^t \phi(s)ds$ be a Young function. Then, $A_{\phi}$ holds if and only if
\begin{equation}\label{New Aphi condition by Kerman}
\frac{ \int_0^{t_1} \phi^{-1}(s^{-\gamma})ds    + \int_0^{t_2} \phi^{-1}(s^{-\gamma})ds    }{t_1 + t_2} \leq \phi^{-1}(Ct_2^{-\gamma}),
\end{equation}
for some $C>1$ independent of $0 \leq t_1 < t_2$.

If further, one has (\ref{Delta two for H}), then (\ref{New Aphi condition by Kerman})can be replaced by (\ref{our condition for Hilbert transform}).
\end{lemma}
\begin{proof}
The condition $A_{\phi}$ for the weight $w(x) = |x|^{\gamma}$ reads
\begin{equation}\label{Aphi for gamma power weight}
\frac{\epsilon   \mu_{\gamma}(I)}{|I|}  \phi  \left( \frac{1}{|I|} \int_I \phi^{-1} \left(   \frac{1 }{\epsilon |x|^{\gamma}}  \right) dx \right)     \leq C,
\end{equation}
for all intervals $I \subset \mathbb{R}$ and $\epsilon >0$. Given $I \, {=} \, [a,b]$, the change of variable $y    \, {=} \,  \epsilon^{\frac{1}{\gamma}}x$ in the integrals  $\int_{I} \epsilon  |x|^{\gamma}dx$ and $\int_I \phi^{-1} \left(   \frac{1 }{\epsilon |x|^{\gamma}}  \right) dx $ yields $ \epsilon^{-\frac{1}{\gamma}} \int_{I_{\epsilon}} |y|^{\gamma}dy$ and $\epsilon^{-\frac{1}{\gamma}} \int_{I_{\epsilon}} \phi^{-1} \left(   \frac{1 }{|y|^{\gamma}}  \right) dy$, respectively, where $I_{\epsilon}  \, {=} \,   [   \epsilon^{\frac{1}{\gamma}} a,  \epsilon^{\frac{1}{\gamma}}b  ]$. So $A_{\phi}$ becomes

\begin{equation*}
\frac{   \mu_{\gamma}(I_{\epsilon})}{|I_{\epsilon}|}  \phi  \left( \frac{1}{|I_{\epsilon}|} \int_{I_{\epsilon}} \phi^{-1} \left(   \frac{1 }{|y|^{\gamma}}  \right) dy \right)     \leq C,
\end{equation*}
Since $I_{\epsilon}$ is arbitrary whenever $I$ is, it suffices to verify $A_{\phi}$ with $\epsilon \, {=} \, 1$. As in the proof of \Cref{characterization of $P_p$ and $Q_q$}
\begin{equation*}
\frac{\mu_{\gamma}(I)}{|I|} \approx \max [|a|, |b|]^{\gamma},
\end{equation*}
whence, $A_{\phi}$ asserts
\begin{equation*}
 \phi  \left( \frac{1}{b-a} \int_{a}^{b} \phi^{-1} \left(   \frac{1 }{|y|^{\gamma}}  \right) dy \right)     \leq \frac{C}{ \max [|a|, |b|]^{\gamma}},
\end{equation*}
which is equivalent to (\ref{New Aphi condition by Kerman})  if $a<0<b$. In particular, we have $A_{\phi}$ implies (\ref{New Aphi condition by Kerman}). 

For the converse, it remains to show that, in case of  $ab \geq 0$, (\ref{New Aphi condition by Kerman}) implies $A_{\phi}$. To this end, suppose, then, $ab \geq 0$, say $0 \leq a <b$, so that
(\ref{Aphi for gamma power weight}) is
\begin{equation*}
 \phi  \left( \frac{1}{b-a} \int_{a}^{b} \phi^{-1} \left(   s^{-\gamma}  \right) ds \right)    \leq  C  \  b^{-\gamma}.
\end{equation*}
Since $\phi^{-1}(s^{-\gamma})$ decreases, one has
\begin{equation*}
\frac{1}{b-a} \int_{a}^{b} \phi^{-1} \left(   s^{-\gamma}  \right) ds  \leq \frac{1}{b} \int_{0}^{b} \phi^{-1} \left(   s^{-\gamma}  \right) ds.
\end{equation*}
The condition (\ref{New Aphi condition by Kerman}), with $t_1=0$ and $t_2=b$, and above inequality yield the $A_{\phi}$ condition.

Finally, (\ref{New Aphi condition by Kerman}) always implies (\ref{our condition for Hilbert transform}) - just take $t_1 = 0$. Moreover,
\begin{equation}
\begin{split}
\frac{ \int_0^{t_1} \phi^{-1}(s^{-\gamma})ds    + \int_0^{t_2} \phi^{-1}(s^{-\gamma})ds    }{t_1 + t_2} &  \leq   \frac{2}{t_2} \int_0^{t_2} \phi^{-1}(s^{-\gamma})ds \\
& \leq 2 \  \phi^{-1}(C t_{2}^{-\gamma}),
\end{split}
\end{equation}
ensures (\ref{our condition for Hilbert transform}) when (\ref{Delta two for H}) holds; see \cite{KR}.
\end{proof}

\begin{proof}[Proof of \cref{characterization for H}] The equivalence of $\ref{(H1)}$ and $\ref{(H2)}$ follows from the variant of \cref{gauge norm equivalent modular} for $|x|^{\gamma}$ on $\mathbb{R}$, since $H$ is dilation-invariant.

The condition (\ref{Delta two for H}) in 
(BK) comes out of the inequality in $\ref{(H2)}$ in the same way it comes out of the corresponding inequality for $M$ in Theorem $7$ of \cite{BK1}, but with
\begin{equation*}
(M f_{m})(y) \geq C |E \cap B_{m}| \  |x-y|^{-1}, \ \ \ y \notin B_{m},
\end{equation*}
replaced by
\begin{equation*}
(H f_{m})(y) \geq C r_{0} \  |x-y|^{-1}, \ \ \ y \notin B_{m},
\end{equation*}
where $f_{m}= \chi_{B_{m}}$, $B_{m}= (x-2^{-m} r_{0}, x+ 2^{-m}r_{0})$. Indeed, if, for instance, $y < x- 2^{-m}r_{0}$,

\begin{align*}
-(Hf_{m})(y)= \frac{1}{\pi} \int_{x-2^{-m}r_{0}}^{x+2^{-m}r_{0}} \frac{1}{y-z}dz  & = \frac{1}{\pi} \log \left[ \frac{ x-y -2^{-m}r_{0}}  {x-y + 2^{-m}r_{0}} \right]\\
& = \frac{1}{\pi} \log \left[ 1-  \frac{ 2^{-m}r_{0}}  {x-y + 2^{-m}r_{0}} \right]\\
& \geq \frac{1}{\pi} \ \frac{ 2^{-m}r_{0}}  {x-y + 2^{-m}r_{0}} \\
& \geq \frac{1}{\pi} \ \frac{ 2^{-m-1}r_{0}}  {|x-y|}.
\end{align*}

Again, by Corollary $2.7$ in \cite{BK2}, the modular inequality in $\ref{(H2)}$ is equivalent to
\begin{equation*}
\int_{\mathbb{R}} \Psi \left(  |x|^{-\gamma} |(Hf)(x)| \right) |x|^{\gamma}dx \leq \int_{\mathbb{R}} \Psi(K |y|^{-\gamma} |f(y)|)|y|^{\gamma}dy < \infty,
\end{equation*}
which implies, by the argument above, the condition (\ref{Delta two complement for H}).

Next, the argument in \cite[p. 280]{KT}, applied to $\ref{continuity of Hilbert transform H}$ yields the $A_{\phi}$ condition in (\ref{Aphi for gamma power weight}) for $w(x)= |x|^{\gamma}$, provided one can replace $(Mf)(x)$ in
\begin{equation*}
(Mf)(x) \geq \rho_{\Psi, \epsilon |x|^{\gamma}} (\chi_{I} / \epsilon |\cdot|^{\gamma}) \epsilon \mu_{\gamma}(I)
\end{equation*}
by $|(Hf)(x)|$. In \cite{KT} $f$ was a nonnegative, measurable function supported in $I$, with $\rho_{\Psi, |x|^{\gamma}}(f)=1$ and
\begin{equation*}
\int_{I} f(x)dx = \rho_{\Psi, |x|^{\gamma}} (\chi_{I} / |\cdot|^{\gamma}).
\end{equation*}
But for this $f$ and $x \in I + |I|$, one has
\begin{equation*}
|(Hf)(x)| \geq \frac{1}{2 \pi} \ \rho_{\Psi, |x|^{\gamma}} (\chi_{I} / |\cdot|^{\gamma}) \frac{\chi_{J}(x)}{|I|},
\end{equation*}
and so, as $\Phi$ satisfies the modular inequality in $\ref{(H2)}$,
\begin{align*}
& \int_{J} \Phi \left(   \frac{  \rho_{\Psi, |x|^{\gamma}} (\chi_{I} / |\cdot|^{\gamma})  }      {|I|}    \right) |y|^{\gamma}dy\\
& \leq C \int_{\mathbb{R}} \Phi(|f(y)|)|y|^{\gamma}dy = C;
\end{align*}
that is,
\begin{equation*}
\Phi \left(   \frac{  \rho_{\Psi, |x|^{\gamma}} (\chi_{I} / |\cdot|^{\gamma})  }      {|I|}    \right) \mu_{\gamma}(J) \leq C.
\end{equation*}
Similarly, there holds
\begin{equation*}
\Phi \left(   \frac{  \rho_{\Psi, |x|^{\gamma}} (\chi_{J} / |\cdot|^{\gamma})  }      {|J|}    \right) \mu_{\gamma}(I) \leq C,
\end{equation*}
whence
\begin{equation*}
\Phi \left(   \frac{  \rho_{\Psi, |x|^{\gamma}} (\chi_{J} / |\cdot|^{\gamma})  }      {|J|}    \right) \mu_{\gamma}(J) \ \   \Phi \left(   \frac{  \rho_{\Psi, |x|^{\gamma}} (\chi_{I} / |\cdot|^{\gamma})  }      {|I|}    \right)    \mu_{\gamma}(I) \leq C^2.  
\end{equation*}
To get $A_{\phi}$ (for $\epsilon=1$, which is enough) it suffices to show
\begin{equation*}
\Phi \left(   \frac{  \rho_{\Psi, |x|^{\gamma}} (\chi_{J} / |\cdot|^{\gamma})  }      {|J|}    \right) \mu_{\gamma}(J) \geq 1,
\end{equation*}
or, equivalently,
\begin{equation*}
\frac{1}   {   \Phi^{-1} \left(\frac{1}{\mu_{\gamma}(I)} \right)    }     \  \rho_{\Psi, |x|^{\gamma}} (\chi_{J} / |\cdot|^{\gamma}) \geq |J|,
\end{equation*}
that is,
\begin{equation*}
\rho_{\Phi, |x|^{\gamma}} (\chi_{J} )  \ \  \rho_{\Psi, |x|^{\gamma}} (\chi_{J} / |\cdot|^{\gamma}) \geq |J|,
\end{equation*}
which inequality is essentially the generalized H\"older inequality
\begin{align*}
|J| & = \int_{\mathbb{R}} \chi_{J}(x) \frac{\chi_{J}(x)}{ |x|^{\gamma}} |x|^{\gamma} dx\\
& \leq 2 \rho_{\Phi, |x|^{\gamma}} (\chi_{J} )  \ \   \rho_{\Psi, |x|^{\gamma}} (\chi_{J} / |\cdot|^{\gamma}).
\end{align*}

Finally, we prove (BK)
 implies $\ref{(H2)}$. According to Theorem $7$ in \cite{KT}, $|x|^{\gamma}$ in $A_{\phi}$ together with (\ref{Delta two for H}) and (\ref{Delta two complement for H}), implies $|x|^{\gamma}$ satisfies the $A_{\infty}$ condition, namely, there exist constants $C, \delta > 0$ so that for any interval $I$ and any measurable subset $E$ of $I$,
\begin{equation*}
\frac{\mu_{\gamma}(E)}{\mu_{\gamma}(I)} \leq C  \left( \frac{|E|}{|I|} \right)^{\delta}.
\end{equation*}

The argument on p. $245$ of \cite{CF} then ensures the maximal Hilbert transform, $H^*$, defined at $f \in L_{1} \left( \frac{1}{1+ |y|} \right)$ by
\begin{equation*}
(H^*f)(x)= \sup_{\epsilon >0} \biggl|  \frac{1}{\pi}  \int_{|x-y|> \epsilon} \frac{f(y)}{ x-y}dy  \biggr|, \ \ \ x \in \mathbb{R},
\end{equation*}
satisfies, for any given $\alpha >0$ and the $\delta > 0$ in the $A_{\infty}$ condition,
\begin{equation*}
\int_{ \{ H^*f > 2 \lambda, \  Mf \leq \alpha \lambda \}} |x|^{\gamma} dx \leq C \alpha^{\delta} \int_{ \{   Mf >  \lambda \}} |x|^{\gamma} dx,
\end{equation*}
in which $C>0$ does not depend on $\alpha$, $\lambda$ or $f \in L_{1} \left( \frac{1}{1+ |y|} \right)$.

We thus have, since $\Phi$ satisfies (\ref{Delta two for H}),
\begin{align*}
\int_{\mathbb{R}} \Phi \left( (H^*f)(x) \right) |x|^{\gamma} dx & = C \int_{\mathbb{R_+}} \phi(\lambda)    \int_{ \{ H^*f > 2 \lambda \}} |x|^{\gamma} dx  \  d \lambda \\
& \leq C \int_{\mathbb{R_+}} \phi(\lambda)    \int_{ \{ Mf > \alpha \lambda \}} |x|^{\gamma} dx  \  d \lambda     +   C \alpha^{\delta} \int_{\mathbb{R_+}} \phi(\lambda)    \int_{ \{ H^*f >  \lambda \}} |x|^{\gamma} dx  \  d \lambda \\
& = \frac{C}{\alpha} \int_{\mathbb{R_+}} \phi({\lambda / \alpha})    \int_{ \{ Mf >  \lambda \}} |x|^{\gamma} dx  \  d \lambda     +   C \alpha^{\delta} \int_{\mathbb{R_+}} \phi(\lambda)    \int_{ \{ H^*f >  \lambda \}} |x|^{\gamma} dx  \  d \lambda \\
& \leq C' \int_{\mathbb{R_+}} \phi({\lambda })    \int_{ \{ Mf >  \lambda \}} |x|^{\gamma} dx  \  d \lambda     +   C \alpha^{\delta} \int_{\mathbb{R_+}} \phi(\lambda)    \int_{ \{ H^*f >  \lambda \}} |x|^{\gamma} dx  \  d \lambda \\
\end{align*}
Taking $\alpha$ such that $C \alpha^{\delta} < \frac{1}{2}$ we get
\begin{align*}
\int_{\mathbb{R}} \Phi \left( |(Hf)(x)| \right) |x|^{\gamma} dx & \leq \int_{\mathbb{R}} \Phi \left( (H^*f)(x) \right) |x|^{\gamma} dx \\
& \leq K \int_{\mathbb{R}} \Phi \left( (Mf)(x) \right) |x|^{\gamma} dx \\
& \leq  \int_{\mathbb{R}} \Phi \left( K |f(x)| \right) |x|^{\gamma} dx,
\end{align*}
by \cref{characterization for M}, since (\ref{our condition for Hilbert transform}) implies (\ref{our condition for maximal function}).
\end{proof}

\section{Appendix I}

The two general results in this appendix are variants of Theorem 3.1 and  4.1 in \cite{BK2}.

\begin{proposition}\label{weak vs weighted weak modular}
Let $t,u,v$ and $w$ be weights on $\mathbb{R_+}$. Suppose 
$\Phi_1$ and $\Phi_2$ are  nonnegative nondecreasing functions on $\mathbb{R_+}.$ Then, the general weighted modular inequality for 
\begin{equation*}
(If)(x)= \int_0^x f(y)dy ~, 0 \leq f\in M_+(\mathbb{R_+}), x\in \mathbb{R_+},
\end{equation*}
namely,
\begin{equation}\label{eq:M}
\int_{\mathbb{R_+}}\Phi_1(w(x)If(x))t(x)dx \leq \int_{\mathbb{R_+}} \Phi_2(Ku(y)f(y))v(y)dy
\end{equation}
is equivalent to the weighted weak-type modular inequality
\begin{equation}\label{eq:WM}
\int_{\{x\in \mathbb{R_+}: (If)(x) > \lambda \}} \Phi_1(\lambda w(x))t(x)dx \leq \int_{\mathbb{R_+}}\Phi_2(Ku(y)f(y))v(y)dy.
\end{equation}
in both of which $K>0,$ is independent of $0 \leq f \in M(\mathbb{R_+})$ and in $(\ref{eq:WM})$ is independent of $\lambda$ as well.
\end{proposition}

\begin{proof} Clearly, $(\ref{eq:M})$ implies $(\ref{eq:WM})$.
To prove the converse fix $f \geq 0$ and choose $x_k$ so that $ If(x_k)=2^k , k =0,\pm 1,\pm 2,\ldots$ and set $I_k=[x_{k-1},x_k)$ and $f_k= f\chi_{I_k}.$
Then,
\begin{equation*}
\begin{split}
\int_{\R} \Phi_1( w(x)If(x)) t(x)dx & = \sum_{k \in \mathbb Z} \int_{I_k} \Phi_1( w(x)If(x)) t(x)dx\\
& \leq \sum_{k \in \mathbb Z}  \int_{I_k} \Phi_1( 2^k w(x)) t(x)dx.\\
\end{split}
\end{equation*}
For  $x \in I_k,$ one has 
\begin{equation*}
I(8f_{k-1})(x)\geq 8 \int_0^{x_{k-1}} f_{k-1}(y) dy=8\int_{x_{k-2}}^{x_{k-1}}f(y)dy=2^{k+1}>2^k,
\end{equation*}
so
\begin{equation*}
 I_k \subset  \{x\in \mathbb{R_+}: I(8 f{k-1})(x) > 2^k\}.
 \end{equation*}
Thus, by $(\ref{eq:WM})$
\begin{equation*}
\begin{split}
\int_{I_k} \Phi_1( 2^k w(x))) t(x)dx 
& \leq \int_ \mathbb{R_+}  \Phi_2(8 K f_{k-1}(x))v(y)dy\\
&=\int_{x_{k-2}}^{x_{k-1}}\Phi_2(Ku(y)f(y))v(y)dy.
\end{split}
\end{equation*}
Altogether, then,

\begin{equation*}
\begin{split}
\int_\mathbb{R_+} \Phi_1( w(x)If(x)) t(x)dx & \leq \sum_{k \in \mathbb Z} \int_{I_k}\Phi_1(2^k w(x))t(x))dx \\
&  \leq  \sum_{k \in \mathbb Z} \int_{\mathbb{R_+}}\Phi_2(8Kf_{k-1}(x)u(x))v(x)dx\\
& = \int_{\mathbb{R_+}}\Phi_2(8K \sum_{k \in \mathbb Z} f_{k-1}(x) u(x))v(x)dx\\
& =  \int_{\mathbb{R_+}}\Phi_2(8 K f(x)u(x))v(x)dx
\end{split}
\end{equation*}
\end{proof}

\begin{proposition}\label{conditions for weak vs weighted weak modular}
Let $t,u,v,w$ and $w$ and
$\Phi_1$ and $\Phi_2$ be as in the Proposition 1. Assume , moreover, that $\Phi_2$ is a Young function. Then, $(\ref{eq:WM})$ (and hence $(\ref{eq:M})$ ) holds if and only if
\begin{equation}
\int_0^x \Psi_2 \left( \frac{\alpha(\lambda,x)}{C \lambda u(y)v(y)} \right)v(y)dy \leq \alpha(\lambda , x) = \int_x^{\infty}\Phi_1(\lambda w(y))t(y)dy< \infty
\end{equation}
or equivalently,
\begin{equation}
\int_0^x \phi^{-1}_2 \left( \frac{\alpha(\lambda,x)}{C \lambda u(y)v(y)} \right) \frac{dy}{v(y)} \leq C,
\end{equation}
in both of which $C>0$ is independent of $\lambda, x \in \mathbb {R_+}$.

\end{proposition}
\begin{proof} Suppose  $(\ref{eq:WM})$  holds and fix $x \in \mathbb {R_+}$.
Since $u$ and $v$ are weights, they are positive a.e. and so
\begin{equation*}
{\Psi_2} \left( \frac{1}{u(y)v(y)} \right)v(y) < \infty , ~~ y \,a.e.
\end{equation*}
Let the set $E_n \subseteq (0,x)$ be such that $E_n \uparrow (0,x)$ and 
\begin{equation*}
\int_{E_n} \Psi_2 \left( \frac{1}{u(y)v(y)} \right)v(y) < \infty, 
\end{equation*}

Given $E_n=E,$ say, consider
\begin{equation*}
\int_E \Psi_2 \left( \frac{\epsilon}{u(y)v(y)} \right)\frac{v(y)}{\epsilon} dy,
\end{equation*}
for $\epsilon \in \mathbb {R_+}.$ Now, $\frac{\Psi_2(t)}{t}$ is continuous, increasing and maps onto $ \mathbb {R_+}$ onto itself, so the same is true of the integral as a function of $\epsilon \in \mathbb {R_+}.$  Thus, given $\lambda \in \mathbb{R_+},$ there exists $\epsilon \in \mathbb {R_+}$ such that
\begin{equation*}
\int_E \Psi_2 \left( \frac{\epsilon}{u(y)v(y)} \right)\frac{v(y)}{\epsilon} =2K \lambda
\end{equation*}

Set
\begin{equation*}
f(y)= \frac{1}{K} \Psi_2 \left( \frac{\epsilon}{u(y)v(y)} \right)\frac{v(y)}{\epsilon}. \chi_E(y)
\end{equation*}
Then, for $s \geq x$,
\begin{equation*}
If(s)\geq If(x)=2\lambda > \lambda.
\end{equation*}
 So,
\begin{equation*}
\begin{split}
\alpha(\lambda,x) & = \int_x^{\infty}\Phi_1(\lambda w(y))t(y)dy\\
& \leq \int_{\{ If> \lambda \}}\Phi_1(\lambda w(y))t(y)dy \\
& \leq \int_{\mathbb {R_+}}\Phi_2(Ku(y)f(y))v(y)dy \\
& = \int_Ev(y)\Phi_2 \left( \frac{u(y)v(y)}{\epsilon} \Psi_2\left( \frac{\epsilon}{u(y)v(y)} \right) \right)dy\\
& \leq \int_E \Psi_2 \left( \frac{\epsilon}{u(y)v(y)} \right)v(y)dy = 2K \lambda \epsilon < \infty.
\end{split}
\end{equation*}
In the last inequality we used the fact that for any Young function $ \Phi$ with complementary function $\Psi$ one has 
$\Phi(\frac{\Psi(t)}{t}) \leq \Psi(t), t \in {\mathbb {R_+}}.$

The forgoing implies (since $\phi^{-1}_2(t) \leq \frac{\Psi_2(2t)}{t})$
\begin{equation*}
\begin{split}
J= & \int_E \phi^{-1}_2 \left(  \frac{\alpha(\lambda,x)}{4K \lambda u(y)v(y)}\right)\frac{dy}{v(y)} \leq  \int_E \phi^{-1}_2 \left(  \frac{\epsilon}{2u(y)v(y)}\right)\frac{dy}{v(y)} \\
& \leq 2 \int_E \Psi_2 \left(  \frac{\epsilon}{u(y)v(y)}\right)\frac{v(y)dy}{\epsilon}= 4K \lambda.
\end{split}
\end{equation*}

But, again, using $\phi^{-1}_2(t) \geq \frac{\Psi_2(t)}{t},$

\begin{equation*}
\begin{split}
J\geq & \frac{4K\lambda}{\alpha(\lambda,x)}\int_E \Psi_2 \left(  \frac{\alpha(\lambda,x)}{4K \lambda u(y)v(y)}\right) v(y)dy.
\end{split}
\end{equation*}

Combining these two inequalities involving $J$ and letting $n \to \infty$ we obtain

\begin{equation*}
\int_0^x \Psi_2 \left(  \frac{\alpha(\lambda,x)}{C \lambda u(y)v(y)}\right)v(y)dy  \leq \alpha(\lambda,x),
\end{equation*}
with $C=4K.$

Conversely, assume we have $(\ref{eq:M})$. Fix $0 \leq f\in M(\mathbb{R_+})$ and $\lambda>0.$
Put  
$E_{\lambda}=\{x \in \mathbb{R_+}:  If(x)> \lambda \}.$ Assume, without loss of generality that $E_{\lambda} \neq \varnothing.$\\
Now $If(x)= \int_0^x f$ is increasing so, without loss of generality we can assume $\int_0^{\infty}f= \infty.$\\
Then $ E_{\lambda}$ has the the form $(\beta, \infty)$ or $[ \beta,\infty).$ 

Let $\gamma > \beta,$ so that $If(\gamma)> \lambda,$ whence

\begin{equation*}
\begin{split}
2\alpha(\lambda,\gamma) & \leq 2\alpha(\lambda,\gamma).\frac{If(\gamma)}{\lambda}= \int_0^{\gamma} 2Kf(y)u(y)\frac{\alpha(\lambda,\gamma)}{\lambda Ku(y)v(y)}.v(y)dy\\
& \leq \int_0^{\gamma} \Phi_2 \left( 2Kf(y)u(y) \right) v(y)dy
+  \int_0^{\gamma} \Psi_2 \left( \frac{\alpha(\lambda,\gamma)}{K\lambda u(y)v(y)}\right)v(y)dy\\
& \leq \int_0^{\gamma} \Phi_2 \left( 2Kf(y)u(y) \right) v(y)dy + \alpha(\lambda,\gamma).
\end{split}
\end{equation*}
So,
\begin{equation*}
\alpha(\lambda,\gamma) \leq \int_0^{\gamma} \Phi_2 \left( 2Kf(y)u(y) \right) v(y)dy.
\end{equation*}
Letting $\gamma \rightarrow \beta$, we have
\begin{equation*}
\int_{\{If > \lambda \}}\Phi_1(\lambda w(x))t(x)dx \leq \int_{\mathbb{R_+}}\Phi_2(Ku(y)f(y))v(y)dy,
\end{equation*}
with constant as $2K.$
\end{proof}

\section{Appendix II}
Let $\Phi(t)= \int_0^t \phi(s)ds$, $t \in \mathbb{R_+}$ be a Young function and let $w$ be a weight on $\mathbb{R}^n$. The conditions
\begin{equation}\label{Bloom Kerman condition}
\frac{1}{w(Q)} \int_{Q} \phi^{-1} \left( \frac{1}{C} \frac{\Phi(\lambda)}{\lambda} \frac{w(Q)}{|Q|} \frac{1}{w(x)} \right)dx \leq \Phi(\lambda) \tag{BK}
\end{equation}
and
\begin{equation}\label{Torcinski condition}
\frac{\epsilon w(Q)}{|Q|} \phi \left( \frac{1}{|Q|} \int_{Q} \phi^{-1}  \left( \frac{1}{\epsilon w(x)} \right) dx \right) \leq C, \tag{$A_{\phi}$}
\end{equation}
in which $C>1$ is to be independent of $\lambda, \epsilon $ in $\mathbb{R_+}$ and $Q$ is a cube in $\mathbb{R}^n$, $w(Q)= \int_{Q} w(x)dx$, were introduced in \cite{BK1} and \cite{KT}, respectively. To put the two conditions on the same footing we will work with (\ref{Bloom Kerman condition}) in the equivalent form
\begin{equation*}
\frac{1}{|Q|} \int_{Q} \phi^{-1} \left( \frac{1}{C} \phi(\lambda) \frac{w(Q)}{|Q|} \frac{1}{w(x)} \right)dx \leq C \lambda.
\end{equation*}

Our aim in this section is to compare (\ref{Bloom Kerman condition}) and (\ref{Torcinski condition}) in the context of power weights on $\mathbb{R}$, namely, the conditions (\ref{our condition for Hilbert transform})  and (\ref{New Aphi condition by Kerman}). We have already observed that (\ref{New Aphi condition by Kerman})  implies (\ref{our condition for Hilbert transform}). Indeed, (\ref{Torcinski condition}) implies (\ref{Bloom Kerman condition}) in general, as seen in
\begin{theorem}
Let $\Phi(t)=\int_0^t \phi(s)ds, \ t \in \mathbb{R_+}$, and let $w$ be a weight on $\mathbb{R}^n$. Then, (\ref{Torcinski condition}) implies (\ref{Bloom Kerman condition}). 
\end{theorem}
\begin{proof}
Writing (\ref{Torcinski condition}) in the form
\begin{equation*}
\frac{1}{|Q|} \int_{Q} \phi^{-1} \left( \frac{1}{\epsilon w(x)} \right)dx \leq \phi^{-1} \left( \frac{1}{\epsilon} \frac{C|Q|}{ w(x)} \right),
\end{equation*}
then setting $\frac{1}{\epsilon} = \phi(\lambda) \frac{w(Q)}{C|Q|}$, we obtain
\begin{equation*}
\frac{1}{|Q|} \int_{Q} \phi^{-1} \left( \frac{1}{C} \phi(\lambda) \frac{w(Q)}{|Q|} \frac{1}{w(x)} \right)dx \leq \phi^{-1} \left( \phi(\lambda) \right)= \lambda,
\end{equation*}
which is of course, (\ref{Bloom Kerman condition}).
\end{proof}
We now show that to each power weight $\mu_{\gamma}(x)=|x|^{\gamma}, \gamma >0$, on $\mathbb{R}$ there corresponds a Young function, $\Phi_{\gamma}(t)= \int_0^t \phi_{\gamma}(s)ds$, $t \in \mathbb{R_+}$, for which (\ref{Bloom Kerman condition}) holds, but (\ref{Torcinski condition}) doesn't.

\begin{example}
We define $\Phi_{\gamma}$ in terms of decreasing function $\chi$ as
\begin{equation*}
\phi_{\gamma}(t) = \chi(t^{- \frac{1}{\gamma}}), \ \ \  t \in \mathbb{R_+},
\end{equation*}
where
\begin{equation*}
\chi(t)= \log(e/t), \ \ \ 0<t\leq 1,
\end{equation*}
and
\begin{equation*}
    \chi(t)= \begin{cases}
                 \frac{1}{2^{k}} \left( 1 - \frac{t - a_{k}}{2} \right) & a_{k}<t \leq a_{k}+1,\\
                \frac{1}{2^{k+1}}  & a_{k}+1<t \leq a_{k+1},
                \end{cases}
\end{equation*}
with $a_{0}=1$ and $a_{k}= (k+3)!$, $ k \geq 1$.

If (\ref{Torcinski condition}) held, one would have, on taking $t_{1}=t_{2}=t$ in (\ref{New Aphi condition by Kerman})
\begin{equation*}
\frac{1}{t} \int_0^t \chi(s)ds \leq \chi \left( t/C^{\frac{1}{\gamma}} \right), \ \ \ t \in \mathbb{R_+},
\end{equation*}
for same $C >1$. But, for $k \geq 1$,
\begin{equation*}
\frac{1}{a_{k}} \int_0^{a_{k}} \chi(s)ds \geq \frac{1}{2^{k}}= \chi \left( a_{k}/k \right).
\end{equation*}
It thus suffices to show
\begin{equation*}
\frac{1}{t} \int_0^t \chi(s)ds \leq 4\chi \left( t/4^{\frac{1}{\gamma}} \right), \ \ \ t \in \mathbb{R_+}.
\end{equation*}
This is readily done when $0 < t \leq 1$. For $t \in (a_{k}, a_{k+1}]$, $k \geq 0$, one has
\begin{equation*}
\frac{1}{t} \int_0^t \chi(s)ds = \frac{a_{k}}{t} \frac{1}{a_{k}} \int_0^{a_{k}} \chi(s)ds + \frac{1}{2^{k+1}} \left[ \frac{3}{2t} + 1- \frac{( a_{k}+1)}{t} \right]
\end{equation*}
If we can prove
\begin{equation}\label{star inequality in example}
\frac{1}{a_{k}} \int_0^{a_{k}} \chi(s)ds \leq 2 \chi \left( a_{k} \right) \ \ \ \text{for each} \ k, \tag{$\ast$}
\end{equation}
then the above gives
\begin{equation*}
\begin{split}
\frac{1}{t} \int_0^t \chi(s)ds & \leq \frac{1}{2^{k}} \left( \frac{2a_{k}}{t} + 1- \frac{ a_{k}}{t} \right)\\
& = \frac{1}{2^{k}} \left( 1+ \frac{ a_{k}}{t} \right)\\
& \leq \frac{4}{2^k}\\
& \leq 4 \chi \left( t/4^{\frac{1}{\gamma}} \right).
\end{split}
\end{equation*}
We prove (\ref{star inequality in example}) by induction. It is readily shown for $k=0$. Assuming it holds for $k$, we prove it for $k+1$.

Indeed,
\begin{equation*}
\begin{split}
\frac{1}{a_{k+1}} \int_{0}^{a_{k+1}} \chi(s)ds & = \frac{a_{k}}{a_{k+1}} \frac{1}{a_{k}} \int_{0}^{a_{k}} \chi(s)ds + \frac{1}{a_{k+1}} \int_{a_{k}}^{a_{k+1}} \chi(s)ds \\
& = \frac{a_{k}}{a_{k+1}} 2 \chi(a_{k}) + \frac{1}{a_{k+1}} \int_{a_{k}}^{a_{k+1}}   \frac{1}{2^{k}} \left( 1 - \frac{s - a_{k}}{2} \right)ds \\
& \leq \frac{a_{k}}{a_{k+1}} \frac{1}{2^{k-1}} + \frac{1}{a_{k+1}} \frac{1}{2^{k}} \\
& \leq \frac{1}{k+4} \frac{1}{2^{k-1}} + \frac{1}{(k+4)!} \frac{1}{2^{k}} \\
& \leq \frac{5}{24}\frac{1}{2^{k}}\\
& < \frac{1}{2^{k}} = 2 \chi(a_{k+1}).
\end{split}
\end{equation*}
\end{example}

In view of \cite[Theorem 1]{M} and \cite[Theorem 1]{BK1}, (\ref{Torcinski condition}) and (\ref{Bloom Kerman condition}) are equivalent if $\Psi(2t) \leq C \Psi(t), \ t \in \mathbb{R_+}$, that is $\Psi \in \Delta_2$. Moreover, one can show this is also the case if $\Phi \in \Delta_2$. However, neither $\Psi \in \Delta_2$ nor $\Phi \in \Delta_2$ is necessary for the equivalence of  (\ref{Torcinski condition}) and (\ref{Bloom Kerman condition}), since both conditions hold for \emph{all} Young functions when $w(x) \equiv 1$.


\begin{thebibliography}{7}
\bibitem[BK1]{BK1} Bloom, S. and Kerman, R., \emph{Weighted Orlicz space integral inequalities for the Hardy-Littlewood maximal operator}. Studia Mathematica 110.2 (1994), 149-167.
\bibitem[BK2]{BK2} Bloom, S. and Kerman, R., 
\emph{Weighted $L_{\Phi}$ integral inequalities for operators of Hardy type}. Studia Mathematica 110.2 (1994), 149-167.  
\bibitem[BS]{BS}  Bennett, C.  and Sharpley, R., \emph{Interpolation of operators}, Pure and Applied Mathematics, vol.
129, Academic Press Inc., Boston, MA, 1988.
\bibitem[CF]{CF} Coifman, R. and Fefferman, C., \emph{Weighted norm inequalities for maximal functions and singular integrals.} Studia Mathematica 51.3 (1974), 241-250. 
\bibitem[KP]{KP} Kerman, R. and Pick, L., \emph{Explicit formulas for optimal rearrangement-invariant norms in Sobolev imbedding inequalities}, Studia Math 206.2,  (2011), 97-119.
\bibitem[KR]{KR} Krasnoselski\u{i}, M. A.  and Rutitski\u{i}, Ya. B., \emph{Convex functions and Orlicz spaces}, P. Noordhoff Ltd., Groningen, 1961.
\bibitem[KT]{KT} Kerman, R. and Torchinsky, A., \emph{Integral inequalities with weights for the Hardy maximal function}. Studia Mathematica 71.3 (1982), 277-284. 
\bibitem[M]{M} Maligranda, L., \emph{Orlicz spaces and interpolation}, Sem. Math. 5, Dep. Mat., Univ. Estadual de Campinas, Campinas SP, Brazil (1989).
\end{thebibliography}
\end{document}